\documentclass[12pt]{article}
\usepackage{amsmath,amsfonts,amssymb,amsthm,epsfig,epstopdf,titling,url,array}
\usepackage{eqnarray}
\usepackage{enumerate}
\usepackage[a4paper]{geometry}
\usepackage{amscd}
\usepackage{centernot}
\usepackage{authblk}
\theoremstyle{plain} 
\newtheorem{thm}{ Theorem}[section]
\newtheorem{lem}[thm]{Lemma}
\newtheorem{prop}[thm]{Proposition}
\newtheorem{cor}[thm]{Corollary}
\newtheorem{defn}[thm]{Definition}
\newtheorem{ex}[thm]{Example}

\newtheorem* {note*}{Note}

\newcommand{\co} {\mathbb{C}}
\newcommand{\R} {\mathbb{R}}

\newcommand{\iy} {\infty}
\newcommand{\N} {\mathbb{N}}
\newcommand{\D} {\mathbb{D}}

\newcommand{\al}{\alpha}
\newcommand{\lm}{\lambda}
\newcommand{\si}{\sigma}
\newcommand{\om}{\omega}
\newcommand{\eps}{\epsilon}

\newcommand{\h}{\mathcal H}
\newcommand{\A}{\mathcal A}
\newcommand{\B}{\mathcal B}
\newcommand{\BH}{\mathcal B(\mathcal H)}

\newcommand{\norm}[1]{\left\Vert#1\right\Vert}
\newcommand{\abs}[1]{\left\vert#1\right\vert}
\newcommand{\set}[1]{\left\{#1\right\}}

\newcommand{\brc}[1]{\left(#1\right)}
\newcommand{\iner}[1]{\left<#1\right>}

\newcommand{\LN}{\ell^{2}(\mathbb N)}
\newcommand{\LNp}{\ell^{p}(\mathbb N)}

\newcommand{\Dinf}{\D C_\iy(\h)}

\title{\bf \Large Operators with Diskcyclic Vectors Subspaces}
\bf
\author[1]{\bf\footnotesize Nareen Bamerni \thanks{nareen\_bamerni@yahoo.com}}
\author[2]{\bf Adem K{\i}l{\i}\c{c}man \thanks{akilicman@yahoo.com}}
\affil[1,2]{\bf Department of Mathematics, University Putra Malaysia,
43400 UPM, Serdang, Selangor, Malaysia}

\begin{document}
\date{}
\maketitle
\begin{abstract}
In this paper, we prove that if $T$ is diskcyclic operator then the closed unit disk multiplied by the union of the numerical range of all iterations of $T$ is dense in $\h$. Also, if $T$ is diskcyclic operator and $|\lm|\le 1$, then $T-\lm I$ has dense range. Moreover, we prove that if $\al >1$, then $\frac{1}{\al}T$ is hypercyclic in a separable Hilbert space $\h$ if and only if $T \oplus \al I_{\co}$ is diskcyclic in $\h \oplus \co$. We show at least in some cases a diskcyclic operator has an invariant, dense linear subspace  or an infinite dimensional closed linear subspace, whose non-zero elements are diskcyclic vectors. However, we give some counterexamples to show that not always a diskcyclic operator has such a subspace.  \\

{\scriptsize
}
\end{abstract}
\section{Introduction}
An operator $T$ is called {\bf hypercyclic} if there is a vector $x\in \h$ such that $Orb(T,x)=\{T^n x: n\in \N\}$ is dense in $\h$, such a vector $x$ is called {\bf hypercyclic} for $T$. In 1969, Rolewicz  \cite{Rolewicz} constructed the first example of hypercyclic operator in a Banach space. He proved that if $B$ is a backward shift on the Banach space $\LNp$ then $\lm B$ is hypercyclic for any complex number $\lm;\,|\lm|>1$. This led Hilden and  Wallen \cite{Hilden} to consider the scaled orbit of an operator. An operator $T$ is {\bf supercyclic} if there is a vector $x\in \h$ such that $\co Orb(T,x)=\{\lm T^n x: \lm \in \co,\, n\in \N\}$ is dense in $\h$, where $x$ is called {\bf supercyclic vector}. For more information on  hypercyclicity and supercyclicity concepts, one may refer to \cite{dynamic,Erdman,Kitai}.\\

By Rolewicz example, a backward shift $\lm B$ is not hypercyclic whenever $|\lm| \le 1$. In the last case, we can notice that even the multiplication of the closed unit disk $\D=\set{x\in \co:|x|\le 1}$ by the orbit of $B$ will not be dense. Therefore, one may ask  `` Can the multiplication of the closed unit disk by the orbit of an operator be dense?"  In 2003, Zeana \cite{cyclic} considered the disk orbit of an operator. An operator $T$ is called {\bf diskcyclic} if there is a vector $x\in \h$ such that the disk orbit $\D Orb(T,x)=\{\al T^nx:n\geq 0, \al\in \co ,|\al|\leq 1\}$ is dense in $\h$, such a vector $x$ is called {\bf diskcyclic for $T$}. She proved that the diskcyclicity is a mid way between the hypercyclicity and the supercyclicity. 

\[ \begin{array}{ccccccc}
\rm{ Hypercyclicity} &\Rightarrow& \rm{ Diskcyclicity} &\Rightarrow& \rm{Supercyclicity}.
  \end{array} \]
In this paper, all Hilbert spaces are infinite dimensional (unless stated otherwise) separable over the field $\co$ of complex numbers. The set of all diskcyclic operators in a Hilbert space $\h$ is denoted by  $\D C(\h)$ and the set of all diskcyclic vectors for an operator $T$  is denoted by $\D C(T)$.\\	

We recall the following facts from \cite{m3}.
 
 \begin{thm}[ Diskcyclic Criterion]\label{dc}
Let $T\in\BH$. Assume that there exist an increasing sequence of integers $\set{n_k}$, two dense sets $X,Y \subset \h$ and a sequence of maps $S_{n_k} : Y \to \h$ such that: 
\begin{enumerate}
\item $\lim_{k\to\infty}\norm{T^{n_k}x}\norm{S^{n_k}y}=0$ for all $x\in X$, $y\in Y$.
\item $\lim_{k\to\infty}\norm{S_{n_k}y} \to 0$ for all $y \in Y$;
\item $T^{n_k} S_{n_k}y \to y$ for all $y \in Y$.
\end{enumerate}
Then $T$ has a diskcyclic vector.
\end{thm}

\begin{prop}\label{2} Let $ T,\, S\in \BH$ such that $ST=TS$ and $R(S)$ is dense in $\h$. If $x\in \D C(T)$, then $Sx\in \D C(T)$.
\end{prop}

\begin{prop}\label{huge}
If $x$ is a diskcyclic vector of $T$, then $T^nx$ is also a diskcyclic vector of $T$ for all $ n\in \N$.
\end{prop}

\begin{cor}\label{3}
If $T$ is a diskcyclic operator on a Hilbert space $\h$, then the set of all diskcyclic vectors for $T$ is dense in $\h$.
\end{cor}

\begin{prop} \label{2.16}
Let $T\in\D C(\h)$. Then $T^*$ has at most one eigenvalue and that one has
modules greater than $1$.
\end{prop}

\begin{cor}\label{Back not D}
A multiple of a unilateral backward shift on $\LN$ is hypercylcic if and only if it is diskcyclic.
\end{cor}

This paper consists of three sections. In section two, we show that if an operator $T$ is diskcyclic, then the closed unit disk multiplied by the union of the numerical range of all iterations of $T$ is dense in $\co$. We show that $T-\lm I$ has dense range for all $\lm \in \co;\, |\lm|\le 1$  whenever $T$ is diskcyclic. We give a relation between a hypercyclic operator on a Hilbert space $\h$ and a diskcyclic operator on the Hilbert space $\h \oplus \co$. In particular, we show that if $\al >1$, then $\frac{1}{\al}T$ is hypercyclic if and only if $T \oplus \al I_{\co}$ is diskcyclic. Moreover, we give another diskcyclic criterion with respect to a sequence $\set{\lm_{n_k}};\, |\lm_{n_k}|\le 1$, which is equivalent to the main diskcyclic criterion Theorem \ref {dc}.\\

In section three, we show that if $T$ is a diskcyclic operator and $\si_p(T^*)=\phi$, then $T$ has an invariant, dense subspace whose non-zero elements are diskcyclic vectors for $T$. However, we give the counterexample \ref{notvecsub} to show that not all diskcyclic operators must have such a subspaces. Moreover, we show that in some cases a diskcyclic operator may have an infinite dimensional closed subspace whose non-zero elements are diskcyclic vectors for $T$. Particularly, we define the condition $\mathcal B_0$ and use it to show that whenever a diskcyclic operator satisfies the condition $\mathcal B_0$ and diskcyclic criterion, then there is an infinite dimensional closed subspace whose non-zero elements are diskcyclic vectors for $T$. In a parallel with supercyclic operators, we show that if an operator satisfies the diskcyclic criterion and there is a normalized basic sequence $u_n$ goes to zero as $n$ goes to infinity, then there is an infinite dimensional closed subspace whose non-zero elements are diskcyclic vectors for $T$. However, Example \ref{notinfinte} shows that not every diskcyclic operator has such a subspace.

\section{Diskcyclic operators}

To prove our first result we need the following lemma

\begin{lem}\label{1}
A vector $x\in \D C(T)$ if and only if $\frac{x}{\norm{x}}\in \D C(T)$
\end{lem}
\begin{proof}
The proof is clearly follows from the fact $\D Orb(T,\frac{x}{\norm{x}})=\frac{1}{\norm{x}}\D Orb(T,x)$ 
\end{proof}

The numerical range of an operator $T$ is defined as $\om(T)=\set{\iner{Tx,x}:\norm{x}=1}$. 

\begin{thm}
Suppose that $T\in \D C(\h)$. Then
\begin{enumerate}
\item $\D\bigcup_{n=0}^{\iy}\iner{T^nx,x}$ is dense in $\co$ for all vectors $x\in \D C(T)$. 
\item $\D \bigcup_{n=0}^{\iy} \om(T^n)$ is dense in $\co$.
\end{enumerate}
\end{thm}
\begin{proof}
$(1)$: Let $x\in \D C(T)$ and $\lm \in \co$. By lemma \ref{1} we can suppose that $\norm{x}=1$. Since $\lm x \in \h$, then there exist an increasing sequence $n_k$ of non-negative integers and a sequence $\al_k \in \co;\, |\al_k|\le 1$ such that 
$$\norm{\al_k T^{n_k}x-\lm x}<\eps $$
Now, 
\begin{eqnarray*}
\abs{\iner{\al_k T^{n_k}x,x}-\lm}&=& \abs{\iner{\al_k T^{n_k}x,x}-\lm\iner{x,x}}\\
&=& \abs{\iner{\al_k T^{n_k}x-\lm x,x}}\\
&\le& \norm{\al_k T^{n_k}x-\lm x}\norm{x}\le \eps
\end{eqnarray*}
It follows that $\set{\D\iner{T^nx,x}:n\ge 0}$ is dense in $\co$.\\

$(2)$: Let $x_0\in \D C(T)$ with $\norm{x_0}=1$ then by $(1)$, $\set{\D\iner{T^nx_0,x_0}:n\ge 0}$ is dense in $\co$.  Since $\D \bigcup_{n=0}^{\iy} \om(T^n)=\D\set{\iner{T^nx,x}:\norm{x}=1 \mbox { and } n\ge 0}$. It follows that $\D \bigcup_{n=0}^{\iy}\om(T^n)$ is dense in $\co$.
\end{proof}

\begin{prop}
If $T\in \D C(\h)$ and $\lm \in \co;\, |\lm|\le 1$, then $T-\lm I$ has dense range.
\end{prop}
\begin{proof}
Suppose that the range of $T-\lm I$  is not dense in $\h$, then there exists $x_0\in \D C(T)$ such that $x_0 \notin \overline{(T-\lm I)\h}$; otherwise $(T-\lm I)\h$  would be dense by Corollary \ref{3}. By the Hahn Banach Theorem there exists a continuous linear functional $f$ on $\h$ such that $f(x_0) \neq 0$ and $f\brc{\overline{(T-\lm I)\h}}=\set{0}$. Then for all $x\in \h, f(Tx)=\lm f(x)$ and so  $f(T^nx)=\lm^n f(x)$ for all $n\in \N$. In particulat, $f(T^nx_0)=\lm^n f(x_0)$. Since $x_0 \in \D C(T)$, then there exist $n_k \to \iy$ and $\al_k \in \co;\, |\al_k| \le 1$ for all $k \in \N$ such that $\al_k T^{n_k}x_0 \to 2x_0$; therefore $\al_k f(T^{n_k}x_0) \to 2f(x_0)$ and hence $\al_k \lm^{n_k} f(x_0) \to  2f(x_0)$. However, since $|\lm|\le 1$  and  $f(x_0)\neq 0$, then $\al_k $ should be greater than $1$ for some $k\in \N$ which is contradiction. 
\end{proof}

By \cite[p.38]{probook}, $\si_p(T^*)=\Gamma (T)$ where $\Gamma (T)$ is the compression spectrum of $T$ i.e the set of all complex numbers $\lm$ such that the range of $T-\lm I$ is not dense. Now, if $T\in \D C(\h)$ and $\lm \in \si_p(T^*)$, then by the last proposition $|\lm| > 1$, which gives another proof of Proposition \ref{2.16}.\\

\begin{thm}\label{5}
 If $T\in \BH$ and $\al$ is a real number such that $\al >1$, then the operator $S=T\oplus \al I_\co \in \mathcal B(\h\oplus \co)$ is diskcyclic if and only if $\frac{1}{\al}T$ is hypercyclic.
\end{thm}
\begin{proof}
Let $z$ be a hypercyclic vector for $\frac{1}{\al}T$, we will show that $z\oplus 1$ is diskcyclic vector for $S$. Let $w\oplus \lm$ be an arbitrary vector in $\h \oplus \co$ with $\lm\neq 0$. Since $\frac{1}{\al}T$ is hypercyclic, then there exist an increasing positive sequence $\set{n_k}$ such that
$$\norm{\brc{\frac{1}{\al}T}^{n_k}z-\frac{1}{\lm}w}\to 0 \qquad as \qquad k\to \iy$$
Therefore,
$$\norm{\lm\brc{\frac{1}{\al}}^{n_k}S^{n_k}\brc{z\oplus 1}-w\oplus \lm}\to 0 \qquad as \qquad k\to \iy$$
and since $\al>1$, then $\lm\brc{\frac{1}{\al}}^{n_k}<1$ as $k\to \iy$.\\
If $\lm=0$, then we can find a sequence $\set{n_k}$ such that 
$$\norm{\brc{\frac{1}{\al}T}^{n_k}z-kw}\to 0 \qquad as \qquad k\to \iy$$
and then
$$\norm{\frac{1}{k}\brc{\frac{1}{\al}}^{n_k}S^{n_k}\brc{z\oplus 1}-w\oplus 0}\to 0 \qquad as \qquad k\to \iy$$
and it is clear that $\frac{1}{k}\brc{\frac{1}{\al}}^{n_k}<1$. Therefore, $z\oplus 1$ is diskcyclic vector for $S$.\\
For the other side, since $S$ is diskcyclic then it is supercyclic and the proof follows directly from \cite[Theorem 5.2]{29}. 
 \end{proof}
\begin{cor}\label{6}
If $\al$ is a real number; $\al >1$ and $c\in \co$, then $z \oplus c \in DC(T\oplus \al I_\co)$ if and only if $z\in HC(T)$.
\end{cor}

\begin{thm}\label{dc lam}
 Let $T \in \BH $, suppose that there exist an increasing sequence of positive integers $\set{n_k}$, a sequence $\set{\lm_{n_k}}\in \co\backslash\set{0}$ such that $\abs{\lm_{n_k}}\le 1$ for all $k\in \N$, two dense sets $X,Y \subset \h$ and a sequence of maps $S_{n_k} : Y \to \h$ such that:
\begin{enumerate}
\item $\norm{\lm_{n_k}T^{n_k}x} \to 0$ for all $x \in X$;
\item $\norm{\frac{1}{\lm_{n_k}}S_{n_k}y} \to 0$ for all $y \in Y$;
\item $T^{n_k} S_{n_k}y \to y$ for all $y \in Y$.
\end{enumerate}
Then there is a vector $x$ such that $\set{\lm_{n_k}T^{n_k}x}$ is dense in $\h$. In particular, $x$ is diskcyclic vector for $T$.
\end{thm}
\begin{proof} The proof follows by Hypercyclic Criterion \cite[Definition 1.5]{dynamic}.
\end{proof}

If the assumptions of the above theorem hold, we say that $T$ satisfies the Diskcyclic Criterion for the sequence $\set{\lm_{n_k}}$. 
\begin{prop}\label{anoth. seq}
If $T$ satisfies the Diskcyclic Criterion for the sequence $\set{\lm_{n_k}}$, then $T$ also satisfies the Diskcyclic Criterion for the sequence $\set{\al_{n_k}}$ where $\abs{\frac{\al_{n_k}}{\lm_{n_k}}}\to 0$. 
\end{prop} 
\begin{proof}
Let $X,Y$ be two dense sets and $S$ be the right inverse to $T$, then there exist a small positive number $\eps$ and a large positive number $J$ such that $\norm{\lm_{n_k}T^{n_k}x}\le \eps$ and  $\norm{\frac{1}{\lm_{n_k}}S_{n_k}y}\le \eps$ for all $k> J$. Setting $\al_{n_k}=\sqrt{\eps}\lm_{n_k}$, it is clear that $\abs{\al_{n_k}}\le 1$ and the proof follows. 
\end{proof}
\begin{prop}
Both diskcyclic criteria are equivalent.
\end{prop}
\begin{proof}
If $T$ satisfies the diskcyclic criterion with respect to the sequence $\set{\lm_{n_k}}$, then it is clear that the conditions (1) and (3) of Proposition \ref{dc} are satisfied. Now since $\norm{\frac{1}{\lm_{n_k}}S_{n_k}y} \to 0$ for all $y \in Y$ and $1\le \abs{\frac{1}{\lm_{n_k}}}$, then the condition (2) of Proposition \ref{dc} holds.\\
Conversely, suppose that $T$ satisfies the diskcyclic criterion. Fix an $x\in X$ and $y\in Y$ then there exist a small positive number $\eps$ such that
\[\norm{T^{n_k}x}\norm{S^{n_k}y}< \eps^2,\]
and
\[\norm{S^{n_k}y}< \eps.\]
Define 
\[\lm_{n_k}=\frac{1}{\eps}S^{n_k}y,\]
It follows that $\abs{\lm_{n_k}}\le 1$ and $\abs{\frac{1}{\lm_{n_k}}}\norm{S^{n_k}y}< \eps_1$ for a small positive number $\eps_1$. Furthermore,
\[\norm{T^{n_k}x}\norm{S^{n_k}y}=\norm{T^{n_k}x}\abs{\lm_{n_k}}\eps< \eps^2,\]
Thus
\[\abs{\lm_{n_k}}\norm{T^{n_k}x}< \eps.\]
which completes the proof.
\end{proof}

\section{Subspaces of diskcyclic vectors}

\begin{defn}
Let $T\in \D C(\h)$ and let $\A$ be a linear subspace of $\h$ whose non-zero elements are diskcyclic vectors for $T$, then $\A$ is called diskcyclic subspace for $T$.
\end{defn}

\begin{thm}\label{8}
Let $T\in \D C(\h)$ and $\si_p(T^*)=\phi$, then $T$ has an invariant, dense diskcyclic subspace. 
\end{thm}
\begin{proof}
Let $x\in \D C(T)$ and $\A=\set{p(T)x: p \mbox{ is polynomial}}$. It is clear that $\A$ is a linear subspace of $\h$, invariant under $T$, and dense in $\h$ since it contains $\D Orb(T,x)$. If the polynomial $p$ is non-constant then, $p(T)=a(T-\mu_1)(T-\mu_k)\ldots$, where $a\neq 0$ and $\mu_1, \ldots ,\mu_k\in\co$. Since $\si_p(T^*)=\phi$, each operator $T-\mu_i$ has dense range; hence $p(T)$ also has dense range. Morever, it is clear that $p(T)$ commutes with $T$, then by Proposition \ref{2}, $p(T)x\subset \D C(T)$ that is every element in $\A$ is a diskcyclic vector for $T$.
\end{proof}

If $T$ satisfies the diskcyclic criterion, then $T$ satisfies the supercyclic criterion. Therefore $\si_p(T^*)=\phi$ by \cite[Proposition 4.3.]{36}. It follows that if $T$ satisfies the diskcyclic criterion, then by the last theorem; $T$ has an invariant, dense diskcyclic subspace .\\

The next example shows that not all diskcyclic operators have diskcyclic subspaces. 

\begin{ex}\label{notvecsub}
Let $\frac{1}{2}T$ be hypercyclic operator on a Hilbert space $\h$ with a hypercyclic vector $x$. Then by theorem \ref{5},  $T \oplus 2I$ is a diskcyclic operator on $\h \oplus \co$ with a diskcyclic vector $x \oplus 1$. By corollary \ref{6}, we can see that  $x \oplus 2 \in DC(T \oplus 2I)$. Suppose that $\A$ is a diskcyclic  subspace for $T \oplus 2I$. Since $(T \oplus 2I)(x \oplus 1)\in DC(T \oplus 2I)$ by Proposition \ref{huge} and since $\A$ is a subspace, then 
$$x \oplus 2-(T \oplus 2I)(x \oplus 1)=(x-Tx)\oplus 0 \in \A $$
however, it is clear that $(x-Tx)\oplus 0 \notin DC(T \oplus 2I)$. Therefore, there is no subspace whose non-zero elements are diskcyclic vectors for $T \oplus 2I$.   
\end{ex}

Not only  a diskcyclic subspace can be dense and invariant, sometimes it can be infinite dimensional closed. In that cases, we say that  $T\in \Dinf$.\\

Montes-Rodr\'{i}guez and Salas \cite{36} defined the condition $\B_0$ to find a suufficient condition for an operator to have an infinite dimensional closed subspace of supercyclic vectors. In parallel with supercyclicity, we define the condition $\B_1$ and use it to find a sufficient condition for an operator to be in $\Dinf$. 
\begin{defn}
Let $T\in \BH$. Suppose that $T$ satisfies the diskcyclicity Criterion with respect to a sequence $\set{\lm_{n_k}}$. If there is an infinite dimensional closed subspace $\B_1\in \h$ such that $\norm{\lm_{n_k}T^{n_k}z}\to 0$ for every $z\in \B_1$, then we say $T$ satisfies Condition $\B_1$ for the sequence $\lm_{n_k}$. 
\end{defn}

\begin{thm}
Let $T\in \BH$. Suppose that $T$ satisfies the Diskcyclicity Criterion with respect to a sequence $\set{\lm_{n_k}}$. If one of the conditions below satisfies, then $T\in \Dinf$.
\begin{enumerate}
\item $T$ satifies condition $\B_1$;
\item There is an infinite dimensional closed subspace $\A\in \h$ such that $\norm{\lm_{n_k}T^{n_k}z}$ is bounded for all $z\in \A$.
\end{enumerate}  
\end{thm}
\begin{proof}
The proof of (1) follows directly from Theorem \ref{dc lam} and \cite[Theorem 2.2]{33}. For (2), Suppose that $T$ satisfies the Diskcyclicity Criterion with respect to a sequence $\set{\lm_{n_k}}$ and there is a positive real number $M$ such that $\norm{\lm_{n_k}T^{n_k}z}<M$ for all $z\in \A$ and $k\in \N$. By Proposition \ref{anoth. seq}, we have $T$ satisfies the diskcyclic criterion with respect to the sequence $\set{\al_{n_k}}$ where $\abs{\frac{\al_{n_k}}{\lm_{n_k}}}\to 0$. Therefore, we have 
\[\norm{\al_{n_k}T^{n_k}z}=\abs{\frac{\al_{n_k}}{\lm_{n_k}}}\abs{\lm_{n_k}}\norm{T^{n_k}z}\le \abs{\frac{\al_{n_k}}{\lm_{n_k}}}M\to 0\]
Thus, we can say that $T$ satisfies condition $\B_1$ and hence the proof is finished.
\end{proof}
\begin{prop}
Let $T\in \BH$. Suppose that $T$ satisfies the diskcyclicity Criterion and there is a normalized basic sequence $\set{u_m}$ such that $\lim_{m\to \iy}Tu_m=0$, then $T\in \Dinf$.
\end{prop}
\begin{proof}
The proof is similar to that given in \cite[Corollary 3.3.]{36}. Since there is no restriction on the sequence $\set{\lm_{n_k}}$ of scalars, we may suppose that $|\lm_{n_k}|\le 1$ for all $k\in \N$. 
\end{proof}

The proof of the following corollary follows directly from Theorem \ref{5}
\begin{cor}
Suppose that $\al \in \R$ and $\al >1$, then the operator $S=T\oplus \al I_\co \in \mathcal B(\h\oplus \co)$ has infinite dimensional closed subspaces of diskcyclic vectors if and only if $\frac{1}{\al}T$ has an infinite dimensional closed subspace of hypercyclic vectors.
\end{cor}

The following example shows that not every diskcyclic operators belong to $\Dinf$
\begin{ex}\label{notinfinte}
Let $\lm$ be a complex number of modulus greater than $1$ and $B$ be the unilateral backward shift operator. Since $\lm B$ is hypercyclic if and only if it is diskcyclic Corollary \ref{Back not D}, then
\begin{enumerate}
\item there exists an invariant, dense linear subspace of diskcyclic vectors for $\lm B$  \cite[p.8]{28}
\item all closed subspaces of diskcyclic vectors for $\lm B$ are finite dimensional \cite[Theorem 3.4]{33}.
\end{enumerate}
\end{ex}


\end{document}